\documentclass[a4paper,11pt,twoside]{amsart}

\usepackage{amssymb}

\usepackage[english,francais]{babel}

\input xy
\usepackage[all]{xy}

\newtheorem{theorem}{Theorem}[section]
\newtheorem{lemma}[theorem]{Lemma}
\newtheorem{e-proposition}[theorem]{Proposition}
\newtheorem{corollary}[theorem]{Corollary}
\newtheorem{e-definition}[theorem]{Definition\rm}

\newtheorem{theoreme}{Th\'eor\`eme}[section]

\newtheorem{proposition}[theoreme]{Proposition}

\def\og{\leavevmode\raise.3ex\hbox{$\scriptscriptstyle\langle\!\langle$~}}
\def\fg{\leavevmode\raise.3ex\hbox{~$\!\scriptscriptstyle\,\rangle\!\rangle$}}

\newcommand{\CC}{\mathbb{C}}
\newcommand{\PP}{\mathbb{P}}
\newcommand{\QQ}{\mathbb{Q}}
\newcommand{\ko}{{\mathcal O}}
\newcommand{\km}{{\mathcal M}}
\newcommand{\ku}{{\mathcal U}}
\newcommand{\kv}{{\mathcal V}}
\newcommand{\kw}{{\mathcal W}}
\newcommand{\Hilb}{{\mathrm{Hilb}}}

\begin{document}




%

\selectlanguage{english}
\title{Stable pairs on elliptic K3 surfaces}



\author[Marcello Bernardara]{Marcello Bernardara}
\email{marcello.bernardara@uni-due.de}

\address[Marcello Bernardara]{Universit\"{a}t Duisburg-Essen, Fakult\"{a}t f\"{u}r Mathematik, Universit\"{a}tsstr. 2,
45117 Essen (Germany).}

\selectlanguage{english}

\begin{abstract}
We study semistable pairs on elliptic K3 surfaces with a section: we construct a family
of moduli spaces of pairs, related by wall crossing phenomena, which can be studied to describe the birational correspondence
between moduli spaces of sheaves of rank 2 and Hilbert schemes on the surface.
In the 4-dimensional case, this can be used to get the isomorphism between the moduli space and the Hilbert scheme
described by Friedman. 

\medskip
\noindent
R{\tiny \'ESUM\'E}. On consid\`ere sur une surface K3 elliptique avec une section une notion de stabilit\'e pour un couple.
On obtient une famille d'espaces de modules reli\'es par wall crossing, dont l'\'etude permet de
d\'ecrire les correspondances birationnelles entre les espaces de modules des faisceaux stables de rang 2
et les sch\'emas de Hilbert sur la surface. En particulier, en dimension 4, ceci permet de d\'ecrire
l'isomorphisme entre l'espace des module et le sch\'ema de Hilbert demontr\'e par Friedman.

\end{abstract}

\maketitle

\selectlanguage{francais}
\section*{Version abr\'eg\'ee en fran\c{c}ais}
Soit $\pi:S \to \PP^1$ une surface elliptique K3 lisse sur $\CC$ dont toute fibre singuli\`ere
est au plus nodale. Si on consid\`ere sur $S$ un faisceau $V$ de rang 2 et de classes de Chern
$c_1(V)$ et $c_2(V)$, il existe une polarisation, dite $(c_1(V),c_2(V))$-convenable,
pour laquelle $V$ est stable si et seulement si sa restriction \`a la fibre g\'en\'erique de
$\pi$ est stable. Si on demande aussi que le degr\'e de $V$ sur la fibre soit impair,
Friedman montre que l'espace des modules des fibr\'es avec telles classes de Chern
semistable par rapport \`a une polarisation convenable est, lorsque non vide, une vari\'et\'e
lisse projective de dimension paire $2t$ birationnelle \`a ${\mathrm{Sym}}^tJ^{e+1}(S)$, o\`u
$2e+1$ est le degr\'e sur la fibre et $J^d(S)$ d\'enote la surface elliptique
dont la fibre g\'en\'erale est isomorphe \'a l'espace des fibr\'es en droites de
degr\'e $d$ sur la fibre g\'en\'erale de $S$.

Si $\pi: S \to \PP^1$ admet une section $\sigma$ et on d\'enote par $f$ la fibre de $\pi$, on
peut, sans perte de g\'en\'eralit\'e, \'etudier les cas o\`u $c_2=1$ et $c_2=\sigma-tf$ pour $t$ entier
positif. On d\'enotera dans la suite par $M(t)$ l'espace des modules des tels fib\'es stables.
Friedman obtient donc une application birationnelle $\varepsilon$ entre $\Hilb^t(S)$ et $M(t)$, qui
est un isomorphisme pour $t \leq 2$, et conjecture que ceci soit vrai pour tout $t$ (\cite[Conj. III, 4.13]{friedellip}).

Pour \'etudier dans le d\'etail la correspondance birationnelle $\varepsilon$, on \'etudie dans ce papier
des espaces de modules de couples stables sur $S$. Huybrechts et Lehn \cite{huylehnart,huylehnart2} d\'efinissent
une notion de stabilit\'e pour un couple $(V,\alpha)$ o\`u $V$ est un faisceau coh\'erent sur une vari\'et\'e lisse
projective et $\alpha:V \to E_0$ un morphisme vers un faisceau $E_0$ fix\'e. Cette notion d\'epend d'un polyn\^ome
\`a coefficients rationnels et, pour un choix convenable (d'ailleurs, g\'en\'erique), il existe un espace
des modules fin projectif des couples stables.
 
Nous d\'efinissons une condition de stabilit\'e, d\'ependant d'un param\`etre
rationnel positif $\delta$, pour un
couple $(V,\alpha)$, o\`u $V$ est un faisceau de rang 2, $c_1(V)=\sigma - tf$
et $c_2(V) = 1$ et $\alpha: V \to \ko_S(\sigma - f)$. Pour $\delta > t+1/2$
cette condition devient trop stricte et il n'y a donc pas de couple stable.
Pour tout entier $n$, la condition de stabilit\'e ne change pas si $\delta$ est compris
entre ${\mathrm{max}}\{0,n-1/2\}$ et $n+1/2$ et, dans ce cas, tout couple semistable est stable.
On peut donc se ramener \`a l'\'etude d'une famille finie d'espaces de modules
$\km_n$ projectifs pour $n$ entier compris entre $0$ et $t$.

La premi\`ere propri\'et\'e qu'on observe est que le premier espace $\km_0$
de la famille admet une fibration en espaces projectifs au dessus
de l'espace des modules $M(t)$. En effet, un couple $(V,\alpha)$ est $0$-stable
si et seulement si le faisceau $V$ est stable. 
Pour tout $V$ stable la fibre est donc donn\'ee par l'espace
projectif $\PP{\mathrm{Hom}}(V,\ko_S(\sigma-f))$. Dans le cas en question,
un tel espace n'est jamais vide et est r\'eduit \`a un point pour un $V$ g\'en\'erique.
On a donc un morphisme birationnel $\km_0 \to M(t)$.
Comme on peut d\'ecrire toujours un tel $V$ comme extension, il est
facile de v\'erifier que le couple g\'en\'erique $(V,\alpha)$ dans $\km_0$ est
$n$-stable pour tout $n=1,\ldots,t$ et de determiner les couples $(V,\alpha)$ qui sont $0$-stables mais
pas $n$-stables pour $n\geq 1$.  

De l'autre c\^ot\'e, il existe un sous-sch\'ema ferm\'e $\tilde{\km}_t$ dans
$\km_t$ qui admet une fibration en espaces projectifs au dessus du
sch\'ema de Hilbert $\Hilb^t(S)$. En fait lorsque on consid\`ere un couple
$(V, \alpha)$ dans $\km_t$ tel que le noyau de $\alpha$ est localement
libre, on peut d\'ecrire $V$ comme extension
de $\ko_S(\sigma-f)\otimes I_Z$ par $\ko_S((1-t)f)$, o\`u $Z$ est
le conoyau de $\alpha$, et donc un sous-sch\'ema localement intersection
compl\`ete dans $S$ de codimension 2 et longueur $t$. Dans le cas en
question, pour un $Z$ g\'en\'erique, une telle extension existe et est unique,
ce qui donne donc un morphisme birationnel $\tilde{\km}_t \to
\Hilb^t(S)$. Le couple g\'en\'erique $(V,\alpha)$ de $\tilde{\km}_t$
est $0$-stable et on peut d\'ecrire dans $\tilde{\km}_t$ les lieux des couples $(V,\alpha)$
non $n$-stables pour $n\leq t$ par l'\'etude du fibr\'e d\'estabilisant maximal du faisceau $V$. 

On dispose finalement d'une suite d'espaces de modules $\km_n$ pour
$n$ entier compris entre 0 et $t$ et tels que le premier et le dernier
des espaces de la suite admettent un morphisme birationnel respectivement
sur l'espace des modules des fibr\'es stables de rang deux et sur le
sch\'ema de Hilbert. La description des couples $0$-stables et $t$-stables comme
extensions permet de d\'ecrire les lieux d'ind\'etermination de l'application
birationnelle $\varepsilon$ \`a travers les correspondances birationnelles
entre les espaces de modules des couples induites par les wall crossing.

Lorsque on fixe $t=2$, ces transformations birationnelles
peuvent \^etre d\'ecrites
explicitement en appliquant des tranformations \'el\'ementaires au couple
universel $(\kv,A)$ de l'espace des modules $\tilde{\km}_2$, qui, dans ce cas, est lisse
et projectif. Ceci nous
permet de montrer l'existence d'un morphisme injectif $\tilde{\km}_2 \hookrightarrow \km_1$ et
d'un morphisme birationnel $\phi_0:\tilde{\km}_2 \to \km_0$.
On obtient comme corollaire le r\'esultat suivant (\cite[Thm. 4.9]{friedellip}).
\begin{theoreme}[Corollary \ref{isoiscorollary}]
Le morphisme $\phi_0$ induit un isomorphisme $\Hilb^2(S) \stackrel{\simeq}{\to} M(2)$.
\end{theoreme}

Le but de cette note est donc de donner un nouveau regard sur la correspondance
birationnelle $\varepsilon$, en explicitant dans
le cas $t=2$
comment retracer dans ce langage l'isomorphisme d\'ej\`a connu.

\selectlanguage{english}

\section{Introduction}

Let $\pi: S \to \PP^1$ be a complex elliptic K3 surface whose singular fibres have at most nodal singularities. Given
a rank 2 torsion free sheaf $V$ with Chern classes $c_1(V)$ and $c_2(V)$, there exists a polarization, called
$(c_1(V),c_2(V))$-suitable, with respect
to which $V$ is stable if and only if its restriction to the generic fibre is stable. This allows Friedman
\cite{friedellip} to show that, if non empty,
the moduli space of such stable sheaves with odd fibre degree $2e+1$ is smooth, of even dimension $2t$
and birational to ${\mathrm{Sym}}^tJ^{e+1}(S)$, where $J^d(S)$ denotes the elliptic surface whose general fiber
is the set of line bundles of degree $d$ on the general fiber of $S$.

If $\pi: S \to \PP^1$ admits a section $\sigma$ and $f$ denotes the generic fibre, we can restrict to the cases $c_2=1$
and $c_1=\sigma-tf$ for a nonnegative integer $t$ and denote by $M(t)$ the moduli space
of rank 2 stable sheaves with such Chern classes.
Friedman's result gives in this case a birational map $\varepsilon$ between $\Hilb^t(S)$ and $M(t)$, which he shows to be
an isomorphism for $t \leq 2$. This leads to conjecture that this map is an isomorphism
for all $t$ (\cite[Conj. III, 4.13]{friedellip}).
 
In this paper, in order to understand closely the birational correspondence $\varepsilon$ between $\Hilb^t(S)$ and $M(t)$, we consider
stable pairs and their moduli spaces as defined and studied in
\cite{huylehnart,huylehnart2}. We give a definition of a $\delta$-stable pair depending on a rational parameter $\delta$,
which gives rise
to a finite family of moduli spaces related by wall crossing phenomena. The first and the last moduli spaces are
birational respectively to $M(t)$ and $\Hilb^t(S)$ and the wall crossing phenomena accurately describe the locus
of indeterminacy of $\varepsilon$. Indeed, if one consider an element $Z$ of $\Hilb^t(S)$, the Serre construction
defines an extension $V$ which is only generically stable, but for such an extension, any pair $(V,\alpha)$ is $t$-stable.
On the other side, any semistable sheaf $V$ in $M(t)$ can be described as an extension, depending on two codimension
2 subschemes of $S$. For such a $V$, any pair $(V,\alpha)$ is $0$-stable. The birational correspondences
between the moduli spaces can be then described basing upon the codimension 2 subschemes appearing in the the extensions.
In the case $t=2$, a detailed description of such correspondences, based on \cite{friedellip}, allows to
define a birational morphism from an irreducible component of the moduli spaces of $2$-stable pairs (which
turns out to be a blow up of $\Hilb^t(S)$) and the moduli space of $0$-stable pairs inducing an isomorphism
between $\Hilb^2(S)$ and $M(2)$.

\section{Stable pairs on elliptic K3 surfaces}\label{stablepairssection}

If $t$ is a nonnegative integer, we fix $L=\ko_S(\sigma+(t+5)f)$ as a $(\sigma-tf,1)$-suitable polarization.
We say that a sheaf is (semi)stable if it is $\mu$-(semi)stable. 

\begin{e-definition}\label{defstabpaironS}
Let $V$ be a rank 2 coherent sheaf over $S$ with $c_1(V) = \sigma - tf$ and $c_2(V)= 1$, $\alpha:V \to
\ko_S(\sigma -f)$ a morphism and $\delta \in \QQ_{>0}$. The pair $(V,\alpha)$ $\delta$-semistable if
\begin{enumerate}
\item[(i)] ${\mathrm{deg}}G \leq 3/2 - \delta$ for all nontrivial
submodules $G \subset {\mathrm{ker}}(\alpha)$,
\item[(ii)] ${\mathrm{deg}}G \leq 3/2 + \delta$ for all nontrivial submodules
$G \subset V$.
\end{enumerate}
Such a pair is $\delta$-stable if both inequalities hold strictly.
\end{e-definition}

This definition is just a special case of Definition 1.1 in \cite{huylehnart}. Then for any positive
$\delta$ the fine moduli
space of stable pairs with respect to $\delta$ exists ans is projective
\cite{huylehnart,huylehnart2}. We denote it by $\km_{\delta}$.

For any integer $n$, if $\delta$ varies in $({\mathrm{max}}\{0,n-1/2\}, n+1/2)$, the moduli spaces
$\km_{\delta}$ are all isomorphic and all semistable pairs are stable. By condition (i) 
there is no semistable pair with respect to $\delta > t+1/2$. There is then a family $\{\km_n\}_{0\leq n \leq t}$
of nonempty projective moduli spaces related by wall crossing
phenomena which give rise to birational maps.

A pair $(V,\alpha)$ is $0$-stable if and only if $V$ is stable. By \cite{friedellip},
any $V$ in $M(t)$ fits a sequence
\begin{equation}\label{extforV}
0 \longrightarrow \ko_S((1-s)f) \otimes I_{Z_1} \longrightarrow V \longrightarrow
\ko_S(\sigma + (1+s-t)f) \otimes I_{Z_2} \longrightarrow 0,
\end{equation}
where $0 \leq s \leq t$ and $l(Z_1)+l(Z_2)=s$. Then $V$
admits at least one nonzero map to $\ko_S(\sigma-f)$ and $\km _0$
fibres over $M(t)$ with fibres given by $\PP{\mathrm{Hom}}(V,\ko_S(\sigma-f))$. Since such fibre
is never empty and generically one dimensional, $\km_0 \to M(t)$ is birational.

Condition (i) gets stronger as $\delta$ grows. Pairs $(V,\alpha)$ in $\km_0$
which do not belong to $\km_n$ are given by extensions (\ref{extforV}) with $s\geq n$.

On the other side, let $\tilde{\km}_t \subset
\km_t$ be the subscheme whose elements are those pairs $(V,\alpha)$ with ${\mathrm{ker}}(\alpha)$ locally free.
In this case, $V$ is given by an extension
\begin{equation}\label{type1ext}
0 \longrightarrow \ko_S((1-t)f) \longrightarrow V \longrightarrow
\ko_S(\sigma - f) \otimes I_Z \longrightarrow 0,
\end{equation}
with $Z$ in $\Hilb^t(S)$. Moreover such an extension is unique for $Z$ generic, which,
together with the following proposition,
tells us that $\tilde{\km}_t$ is projective and birational to $\Hilb^t(S)$.
\begin{proposition}[\cite{huylehnart}, Corollary 2.14]
The set $\tilde{\km}_t$ is a projective scheme over $\Hilb^t(S)$ with fibre over $Z$ isomorphic to
$\PP{\mathrm{Ext}}^1(\ko_S(\sigma - f)\otimes I_Z, \ko_S((1-t)f))$.
\end{proposition}

Condition (ii) gets stronger as $\delta$ decreases. Indeed, for $\delta=0$, this condition
implies that $V$ has no destabilizing subline bundles, while for $\delta \geq 1$, the sheaf $V$ can have
destabilizing subline bundles.
By \cite[III, Prop. 4.4]{friedellip}, the maximal destabilizing subline bundle is of the form $\ko_S(\sigma-af)$
for some integer $a$, then it is not contained in ${\mathrm{ker}}(\alpha)$.
Pairs $(V,\alpha)$ belonging to $\tilde{\km}_t$ but not to $\km_n$ are then unstable extensions
(\ref{type1ext}) such that the maximal destabilizing subline bundle of $V$ is $\ko_S(\sigma-af)$ with
$a > 1+t-n$.

\section{Stable pairs in the case $t=2$}

Consider pairs $(V,\alpha)$ with $c_1(V) = \ko_S(\sigma-2f)$. We show that the
scheme $\tilde{\km}_2$
is smooth and there is an injective morphism $\tilde{\km}_2 \hookrightarrow \km_1$, and that
there is a birational morphism $\tilde{\km}_2 \to \km_0$,
inducing an isomorphism $\Hilb^2(S) \simeq M(2)$.

If $(V,\alpha)$ is a pair in $\tilde{\km}_2$, then $V$ is given by
\begin{equation}\label{genericextension}
0 \longrightarrow \ko_S(-f) \longrightarrow V \longrightarrow
\ko_S(\sigma - f) \otimes I_Z \longrightarrow 0,
\end{equation}
for $Z$ in $\Hilb^2(S)$.
The
dimension of ${\mathrm{Ext}}^1(\ko_S(\sigma - f) \otimes I_Z, \ko_S(-f))$ is 2  
if $Z$ is in ${\mathrm{Sym}}^2\sigma$ and 1 otherwise (see\cite{friedellip}).
Consider the ideal sheaf $I:=I_{{\mathrm{Sym}}^2 \sigma}$.
It can be shown \cite{mathese} that the projectivization $\PP(I)$ is isomorphic
to $\tilde{\km}_2$. Indeed, up to a twist with a line bundle on $\Hilb^2(S)$, the sheaf $I$ is isomorphic to the
sheaf whose stalks are given by extensions (\ref{genericextension}).
\begin{lemma}[\cite{mathese}, Lemma 2.28]\label{m2tildeistheblowup}
The subscheme $\tilde{\km}_2$ is the blow-up of $\Hilb^2(S)$ along ${\mathrm{Sym}}^2 \sigma$.
\end{lemma}

Let $D_{\sigma}$ be the effective divisor of $\Hilb^2(S)$ which is the closure of the locus
of pairs $\{ p, q \vert p \in \sigma \}$. Let $D$ be the irreducible smooth divisor in $\Hilb^2(S)$ given by
$$D=\{ Z \in \Hilb^2(S) \ \vert \ h^0(\ko_S(f) \otimes I_Z) = 1 \}.$$
An argument for the smoothness of $D$ can be found in \cite{friedellip}.
We denote by $\tilde{D}_{\sigma}$ (resp. $\tilde{D}$) the strict transform
of $D_{\sigma}$ (resp. of $D$) and by $\tilde{G}$ the exceptional divisor of the blow-up.

Studying destabilizing subline bundles for the extension (\ref{genericextension}) as $Z$ varies
in $\Hilb^2(S)$ 
allows us to say whether $V$ appears in a pair belonging to
$\km_n$ for $n < 2$. 
\begin{lemma}[\cite{mathese}, Lemma 2.29]\label{destabilizingsublineandpairs}
Let $(V,\alpha)$ be a pair in $\tilde{\km}_2$.  If
$(V,\alpha)$ belongs to $\tilde{D} \cup \tilde{D}_{\sigma}$, then it is $n$-stable
if and only if  $n = 2$. If $(V,\alpha)$ belongs to $\tilde{G} \setminus (\tilde{D} \cup \tilde{D}_{\sigma})$,
then it is $n$-stable if and only if $n=1,2$. In any other case, $(V,\alpha)$ is $n$-stable for
$n=0,1,2$.
\end{lemma}

On the other side, a pair $(V,\alpha)$ belongs to $\km_0$ if and only if $V$ is stable and $\alpha$
is any morphism $V \to \ko_S(\sigma-f)$. The sheaf $V$ is then an extension (\ref{extforV})
with $l(Z_1)+l(Z_2) \leq 2$. There are four possible extension types (see \cite{friedellip}), which we call type $a$
if $l(Z_2) = 2$, type $b$ if $l(Z_2)=1$ and $l(Z_1)=0$, type $c$ if $l(Z_2)=l(Z_1) = 0$ and type $d$ if $l(Z_2)=0$ and
$l(Z_1)=1$.

The generic stable sheaf is given by a type $a$ (which is indeed of the form (\ref{genericextension})) extension. In this case,
there is a unique choice for $\alpha$ and the pair $(V,\alpha)$ belongs to $\tilde{\km}_2$. In particular, such
pairs form the open complementary of $\tilde{D} \cup \tilde{D}_{\sigma} \cup \tilde{G}$ in $\tilde{\km}_2$.
In the non generic
case, the dimension of ${\mathrm{Hom}}(V,\ko_S(\sigma-f))$ is greater then 1, but any morphism
$V \to \ko_S(\sigma-f)$ factors through the extension map \cite{mathese} and 
we can say for which $n$ extensions of type $b$, $c$ and $d$ are $n$-stable.

\begin{lemma}[\cite{mathese}, Lemmas 3.3, 3.4 and 3.5]
If $V$ is a stable type $c$ extension, then any pair $(V,\alpha)$ is $n$-stable if and only if $n=0$. Moreover,
such extensions form in $M(2)$ a subscheme isomorphic to ${\mathrm{Sym}}^2\sigma$.

If $V$ is a type $b$ extension, then any pair $(V,\alpha)$ is $n$-stable if and only if $n=0,1$.

If $V$ is a stable type $d$ extension, then any pair $(V,\alpha)$ is $n$-stable if and only if $n=0,1$. If $V$ is unstable,
then for any $n$ no pair $(V,\alpha)$ is $n$-stable.
\end{lemma}
 
We are now ready to describe birational morphisms between the spaces of pairs. The main tool is given by elementary
transformations of the universal pair $(\kv, A)$ on $S \times \tilde{\km}_2$.

If $(V,\alpha)$ lies in $\tilde{D} \cup \tilde{D}$, then it does not belong to $\km_1$. We can
perform first an elementary transformation of $\kv$ along $S \times \tilde{D}$
by straightforward generalizing a construction by Friedman \cite[III, Prop. 4.12]{friedellip}.
We get a flat reflexive sheaf $\kv'$ over $S \times \tilde{\km}_2$ such that
if $(V,\alpha)$ belongs to $\tilde{D}$, then $\kv'_{(V,\alpha)}$ is a type $d$ extension,
and in that case it is not stable if and only if $(V,\alpha)$ is in
$\tilde{D} \cap \tilde{D}_{\sigma}$.

We now perform an elementary transformation of $\kv'$ along $S \times \tilde{D}_{\sigma}$ to get
a family of sheaves for pairs in $\km_1$.
For any $(V,\alpha)$ in $\tilde{D}_{\sigma}$, there is a unique morphism $\kv'_{(V,\alpha)} \to \ko_S$ (see \cite{mathese}).
We then have a line bundle $\mathcal L$ on $S$ and a surjective morphism
$$\kv'_{\vert S \times \tilde{D}_{\sigma}} \to \pi_1^* \ko_S \otimes \pi_2^* {\mathcal L}
 \to 0,$$
over $S \times \tilde{D}_{\sigma}$.
Define $\ku$ as the elementary transformation
$$0 \longrightarrow \ku \longrightarrow \kv' \longrightarrow i_* (\pi_1^* \ko_S \otimes \pi_2^*
{\mathcal L}) \longrightarrow 0,$$
where $i$ is the embedding of $S \times \tilde{D}_{\sigma}$ in $S \times \tilde{\km}_2$. By \cite[Prop. A2]{friedellip},
the sheaf $\ku$ is flat and reflexive. If $(V,\alpha)$ belongs to $\tilde{D}_{\sigma}$,
then $\ku_{(V,\alpha)}$ is a type $b$ extension, which is unstable if and only if $Z_2 = q \in \sigma$ (recall that a type $b$
extension is an extension (\ref{extforV}) with $l(Z_2)=1$ and $l(Z_1)=0$). Summarizing (see \cite{mathese}):
\begin{itemize}
\item if $(V,\alpha)$ belongs to $\tilde{D} \setminus \tilde{D}_{\sigma}$, then
${\ku}_{(V,\alpha)}$ is a stable type $d$ extension,
\item if $(V,\alpha)$ belongs to $\tilde{D}_{\sigma}$, then ${\ku}_{(V,\alpha)}$ is
a type $b$ extension and it is unstable if and only if $(V,\alpha) \in\tilde{D}_{\sigma} \cap \tilde{G}$.
\end{itemize}
In any case, ${\ku}_{(V,\alpha)}$ belongs to some
pair in $\km_1$. Moreover, if $(V,\alpha)$ is
in $\tilde{D} \cup \tilde{D}_{\sigma}$, the sheaf $U:=\ku_{(V,\alpha)}$ is uniquely determined
and we have a natural choice for a framing map $\beta:U \to \ko_S(\sigma-f)$.
Indeed, if $(V,\alpha)$ lies in $\tilde{D}$, then the elementary transformation of $\kv$ at that
point is induced by the destabilizing exact sequence
$$0 \longrightarrow \ko_S(\sigma - 2f) \stackrel{\iota}{\longrightarrow} V \longrightarrow {\mathfrak m}_q \longrightarrow 0.$$
Such extension class unquely determines the extension class (see \cite[Prop. A2]{friedellip})
$$0 \longrightarrow {\mathfrak m}_q \longrightarrow \kv'_{(V,\alpha)} \stackrel{\gamma}{\longrightarrow} \ko_S(\sigma -2f)
\longrightarrow 0.$$
The map $\alpha':= \iota \circ \alpha : \ko(\sigma-2f) \to \ko (\sigma -f)$ cannot be zero, because
$\ko (\sigma -2f)$ is not in ${\mathrm{ker}}(\alpha)$. There is then a natural choice of a nontrivial framing for
$U$, namely $\beta:= \gamma \circ \alpha'$. A similar argument works also for the second elementary transformation. 
\begin{theorem}\label{theorembiratm2m1}
There is an injective morphism $\phi_1: \tilde{\km}_2 \hookrightarrow \km_1$.
\end{theorem}
\begin{proof}
The universal sheaf $\ku$ over $S \times \tilde{\km}_2$ defines a morphism because for any $(V,\alpha)$ there
is a unique $1$-stable pair $(\ku_{(V,\alpha)},\beta)$. Injectivity is not straightforward only in $\tilde{D} \cap \tilde{D}_{\sigma}$. 
Fix a point $q \in S$ and let $U$ be a corresponding type $b$ extenison. Such
extensions are parametrized by $\sigma$ (see \cite[Lemma 2.39]{mathese}). Let $V$ be the type $a$ extension
with $Z=(q,p)$ such that $p$ in $\sigma$ corresponds to the extension class of $U$. Then $(V,\alpha)$ is the
unique pair such that $U = \ku_{(V,\alpha)}$.
\end{proof}

If $(V,\alpha)$ is not in $\tilde{G}$,
then $\ku_{(V,\alpha)}$ is stable. If $(V,\alpha)$ is in $\tilde{G}$ then $\ku_{(V, \alpha)}$
is unstable and the maximal destabilizing subline bundle is $\ko_S(\sigma -3f)$, which gives,
for all $(V,\alpha)$ in $\tilde{G}$,
$$0 \longrightarrow \ko_S(\sigma - 3f) \longrightarrow \ku_{(V,\alpha)} \longrightarrow \ko_S (f) 
\longrightarrow 0.$$
We then perform an elementary transformation of $\ku$ along $\tilde{G}$ to get a flat and reflexive sheaf $\kw$
over $S \times \tilde{\km}_2$ such that
if $(V,\alpha)$ is in $\tilde{G}$, then $\kw_{(V,\alpha)}$ is a stable type $c$
extension. Arguing as before, we get a morphism $\tilde{\km}_2 \to \km_0$.

\begin{theorem}\label{theorembiratm0m2}
There is a birational morphism $\phi_0 : \tilde{\km}_2 \to \km_0$ which is an isomorphism over the
open complement of $\tilde{G}$.
\end{theorem}

\begin{corollary}\label{isoiscorollary}
There is an isomorphism $\Hilb^2(S) \simeq M(2)$
\end{corollary}
\begin{proof}
Recall that the locus $\Sigma$ of
stable type $c$ extensions in $M(2)$ is isomorphic to ${\mathrm{Sym}}^2\sigma$.
The map $\phi_0$ induces an isomorphism between $\tilde{\km}_2$ and the blow up
of $M(2)$ along $\Sigma$.
This is obtained just by forgetting the framing map of the image of $\phi_0$.
If we take a point $Z$ in ${\mathrm{Sym}}^2 \sigma$, the fibre $\tilde{G}_Z$ over it corresponds,
under this isomorphism, to a fibre over a single point of $\Sigma$. We then have a birational
map which is a bijection between smooth varieties.
\end{proof}

\section*{Remerciements} 
Cet article est issue d'une partie de ma th\`ese de doctorat \cite{mathese}. Je remercie Arnaud Beauville
pour le temps et l'attention dedi\'es \`a ce travail.


\begin{thebibliography}{00}

\bibitem[1]{mathese}M. Bernardara, {\it Cat\'egories d\'eriv\'ees, espaces de modules}, Ph.D. thesis, Universit\'e
de Nice - Sophia Antipolis (2008). Available on my webpage www.uni-due.de/$\sim$hm0092

\bibitem[2]{friedellip}R. Friedman, {\it Vector bundles and SO(3)-invariants for elliptic surfaces},
J. Am. Math. Soc. 8 (1995), no 1, 29-139.

\bibitem[3]{huylehnart}D. Huybrechts, M. Lehn, {\it Stable pairs on curves and surfaces},
J. Algebr. Geom. 4 (1995), no 1, 67-104.

\bibitem[4]{huylehnart2}D. Huybrechts, M. Lehn, {\it Framed modules and their moduli}, Int. J. Math.
6 (1995), no 2, 297-324.

\end{thebibliography}
\end{document}